\documentclass{amsart}
\usepackage{amsmath, amsthm}
\usepackage{nccmath}
\usepackage{amssymb}
\usepackage{mathrsfs}
\usepackage[all]{xy}

\theoremstyle{plain}
\newtheorem{Thm}{Theorem}[section]
\newtheorem{Lem}[Thm]{Lemma}
\newtheorem{Prop}[Thm]{Proposition}
\newtheorem{Cor}[Thm]{Corollary}

\theoremstyle{definition}
\newtheorem{Def}[Thm]{Definition}
\newtheorem{Conj}[Thm]{Conjecture}

\newtheorem{Rem}[Thm]{Remark}
\newtheorem*{Claim}{Claim}
\newtheorem*{ack}{Acknowledgement}
\newtheorem*{que}{Question}

\newenvironment{notation}[0]{%
  \begin{list}%
    {}%
    {\setlength{\itemindent}{0pt}
     \setlength{\labelwidth}{4\parindent}
     \setlength{\labelsep}{\parindent}
     \setlength{\leftmargin}{5\parindent}
     \setlength{\itemsep}{0pt}
     }%
   }%
  {\end{list}}

\begin{document}

\title{On upper bounds of arithmetic degrees}

\author{Yohsuke Matsuzawa}
\address{Graduate school of Mathematical Sciences, the University of Tokyo, Komaba, Tokyo,
153-8914, Japan}
\email{myohsuke@ms.u-tokyo.ac.jp}

\begin{abstract}
Let $X$ be a smooth projective variety defined over $ \overline{\mathbb Q}$,
and $f\colon X \dashrightarrow X$ be a dominant rational map.
Let $\delta_{f}$ be the first dynamical degree of $f$ and 
$h_{X}\colon X( \overline{\mathbb Q}) \longrightarrow [1,\infty)$ be a Weil height function on $X$
associated with an ample divisor on $X$.
We prove several inequalities which give upper bounds of the sequence $(h_{X}(f^{n}(P)))_{n\geq0}$
where $P$ is a point of $X( \overline{\mathbb Q})$ whose forward orbit by $f$ is well-defined.
As a corollary, we prove that the upper arithmetic degree is less than or equal to the first
dynamical degree; $ \overline{\alpha}_{f}(P)\leq \delta_{f}$.
Furthermore, 
we prove the canonical height functions of rational self-maps exist under certain conditions.
For example, when the Picard number of $X$ is one, $f$ is algebraically stable and $\delta_{f}>1$,
the limit defining canonical height $\lim_{n\to \infty}h_{X}(f^{n}(P))\big/\delta_{f}^{n}$ converges.
\end{abstract}


\maketitle

\section{introduction}
Let $X$ be a smooth projective variety defined over a fixed algebraic closure $ \overline{\mathbb Q}$ of 
the field of rational numbers $ {\mathbb{Q}}$
and $f \colon X \dashrightarrow X$ a dominant rational map defined over $ \overline{\mathbb Q}$.
The (first) dynamical degree $\delta_{f}$ of $f$ is a measure of the geometric complexity of the iterates $f^{n}$ of $f$.
The dynamical degree of a dominant rational self-map on an
arbitrary smooth projective variety over $ {\mathbb{C}}$ is defined by Dinh-Sibony in \cite{ds2, ds} using K\"ahler form on $X$.
The alternating definition is introduced by Diller-Favre in \cite{df} using the linear map $f^{*}$ induced on the Neron-Severi group of $X$.
The first dynamical degree is a birational invariant of $f$ and is an important tool for the
study of dynamics of self-maps of algebraic varieties.

On the other hand, in a study of the asymptotic behavior of the Weil heights of iterations $f^{n}(P)$ where
$P\in X( \overline{\mathbb Q})$ is a point whose $f$-orbit is well-defined, Silverman introduced in \cite{s} the notion of arithmetic degree
of the orbit.
It measures the arithmetic complexity of $f$-orbits.
In \cite{s}, he expects the coincidence of the dynamical degree and the arithmetic degree of a Zariski dense orbit.
A refined version of this conjecture was formulated by Kawaguchi and Silverman in \cite{ks}.
Related topics are studied in  \cite{ks3, ks, ks2, mss, mss2, sano, s, s2}.

In this paper, we give upper bounds of heights of $f^{n}(P)$ in terms of $\delta_{f}$.
The main theorem of this paper is Theorem \ref{main} below which says the arithmetic degrees are bounded by the dynamical degree.
Actually, this theorem is stated as Theorem 1 in \cite{ks}.
However,  the proof of Theorem 1 in \cite{ks} unfortunately contains a mistake (cf.\ Remark \ref{error}). 
In this paper, we give a correct proof of Theorem 1 in \cite{ks}.

Before giving a precise statements of our main results, we recall the definition of the dynamical and
arithmetic degrees.

\begin{flushleft}
{\bf The first dynamical degree}
\end{flushleft}

Let $N^{1}(X)$ be the group of divisors on $X$ modulo numerical equivalence.
Since $X$ is smooth, this is equal to the group of codimension one cycles modulo numerical equivalence.
The group $N^{1}(X)$ is a free $ {\mathbb{Z}}$-module of finite rank.
We write $N^{1}(X)_{ {\mathbb{R}}}$ for $N^{1}(X) {\otimes}_{ {\mathbb{Z}}} {\mathbb{R}}$.
The  pull-back homomorphism $f^{*} \colon N^{1}(X) \longrightarrow N^{1}(X)$ is defined as follows.
Take a resolution of indeterminacy $g \colon Y \longrightarrow X$ of $f$ with $Y$ smooth.
Then $f^{*}D=g_{*}((f\circ g)^{*}D)$ for every $D\in N^{1}(X)$.
This is independent of the choice of the resolution.
%
%
\begin{Def}\leavevmode
\begin{enumerate}
\item For an endomorphism $\varphi$ of a finite dimensional real vector space, 
the maximum of the absolute values of  eigenvalues of $\varphi$ is called the spectral radius of $\varphi$
and denoted by $\rho(\varphi)$.
\item The first dynamical degree $\delta_{f}$ of $f$ is defined as follows:
\[
\delta_{f}=\lim_{n\to \infty}\rho((f^{n})^{*} \colon N^{1}(X)_{ {\mathbb{R}}} \longrightarrow N^{1}(X)_{ {\mathbb{R}}})^{1/n}.
\]
Note that $\delta_{f}\geq1$ since $f$ is dominant and $(f^{n})^{*}$ is a homomorphism of the
$ {\mathbb{Z}}$-module $ N^{1}(X)$.
We refer, e.g., to \cite{dang, df, tru} for basic properties of dynamical degrees. 
\end{enumerate}
\end{Def}

\begin{flushleft}
{\bf The arithmetic degree}
\end{flushleft}

The absolute logarithmic Weil height function on $ {\mathbb{P}}^{N}( \overline{\mathbb Q})$
is a function that measures the arithmetic complexity of the coordinates of points
(see for example \cite{bg,hs,l} for the definition).
If we fix an embedding $X \longrightarrow {\mathbb{P}}^{N}$, we get a height function $h_{X}$ on $X( \overline{\mathbb Q})$.

We write $h_{X}^{+}=\max\{ h_{X}, 1\}$.
Let $I_{f}$ be the indeterminacy locus of $f$.
We want to consider the orbit of a point by $f$, so we set
\[
X_{f}( \overline{\mathbb Q})=\{ P\in X( \overline{\mathbb Q})\mid \text{$f^{n}(P)\notin I_{f}$ for all $n\geq 0$}\}.
\]

 \begin{Def}
 Let $P\in X_{f}( \overline{\mathbb Q})$.
 The arithmetic degree of $P$ is 
 \[
 \alpha_{f}(P)=\lim_{n\to \infty}h_{X}^{+}(f^{n}(P))^{1/n}
 \]
 if the limit exists.
 Since it is not known wether the limit always exists, the following invariants are introduced by
 S. Kawaguchi and J. H. Silverman in \cite{ks}.
 \begin{align*}
 &\overline{\alpha}_{f}(P)=\limsup_{n\to \infty}h_{X}^{+}(f^{n}(P))^{1/n}\\
 &\underline{\alpha}_{f}(P)=\liminf_{n\to \infty}h_{X}^{+}(f^{n}(P))^{1/n}.
 \end{align*}
 These are called the upper and lower arithmetic degrees of $P$
 and do not depend on the choice of the embedding $X \longrightarrow  {\mathbb{P}}^{N}$ (see \cite[Proposition 12]{ks}).
 By definition, $1\leq \underline{\alpha}_{f}(P) \leq \overline{\alpha}_{f}(P)$.
 \end{Def}
 
 In \cite{ks}, Kawaguchi and Silverman proposed the following conjecture.
 
\begin{Conj}\label{ksconj}
Let $P \in X_{f}( \overline{\mathbb Q})$.
\begin{enumerate}
\item The limit defining $\alpha_{f}(P)$ exists.
\item The arithmetic degree $\alpha_{f}(P)$ is an algebraic integer.
\item The collections of arithmetic degrees $\{ \alpha_{f}(Q) \mid Q \in X_{f}( \overline{\mathbb Q}) \}$
is a finite set.
\item If the forward orbit $ \mathcal{O}_{f}(P)=\{ f^{n}(P) \mid n=0,1,2, \dots \}$ is Zariski dense in $X$,
then $\alpha_{f}(P)=\delta_{f}$.
\end{enumerate}
\end{Conj}

For example, this conjecture is proved in the following situations:
\begin{enumerate}
\item $ N^{1}(X)_{\mathbb R}= {\mathbb{R}}$ and $f$ is a morphism \cite{ks3}.
\item $f \colon  {\mathbb{P}}^{N} \dashrightarrow {\mathbb{P}}^{N}$ is a monomial map and $P \in {\mathbb{G}}_{{\rm m}}^{N}( \overline{\mathbb Q})$ \cite{s}.
\item $X$ is a surface and $f$ is a morphism \cite{k, mss}.
\item $X= {\mathbb{P}}^{N}$ and $f$ is a rational map extending a regular affine automorphism \cite{ks3}.
\item $X$ is an abelian variety \cite{ks2, s2}.
\end{enumerate}
When $f$ is a morphism, the first three parts of this conjecture are proved by
Kawaguchi and Silverman in \cite{ks2} (cf.\ Remark \ref{remark on mor case}).
See \cite{ks3, mss, sano, s} for more details about this conjecture.

The main theorem of this paper is the following.

\begin{Thm}\label{main}
Let $f \colon X \dashrightarrow X$ be a dominant rational map defined over $ \overline{\mathbb Q}$.
For any $\epsilon>0$, there exists $C>0$ such that
\[
h_{X}^{+}(f^{n}(P))\leq C(\delta_{f}+\epsilon)^{n}h_{X}^{+}(P)
\]
for all $n\geq0$ and $P \in X_{f}( \overline{\mathbb Q})$.
In particular, for any $P\in X_{f}( \overline{\mathbb Q})$, we have
\[
\overline{\alpha}_{f}(P) \leq \delta_{f}.
\]
\end{Thm}

\begin{Rem}\label{error}
This theorem is stated as Theorem 1 in \cite{ks}, but unfortunately their proof is incorrect.
Precisely, in the proof of Theorem 24 (Theorem 1) in \cite{ks},
the constant $C_{1}$ and therefore $C_{8}$ depends on $m$.
Thus one can not conclude the equality $\lim_{m\to \infty}(C_{8}rm^{r})^{1/ml}=1$
which is a key in the argument of the proof in \cite{ks}.
\end{Rem}

\begin{Rem}
We can also define the arithmetic degrees over the one dimensional function field $ \overline{k(t)}$ of characteristic zero.
In \cite{mss2}, Sano, Shibata and I give another proof of  the inequality $ \alpha_{f}(P)\leq \delta_{f}$ over $ \overline{k(t)}$.
\end{Rem}

If $f$ is a morphism, we have the following slightly stronger inequalities.

\begin{Thm}\label{main 1}
Let $f \colon X \longrightarrow X$ be a surjective morphism.
Let $r=\dim N^{1}(X)_{\mathbb R}$ be the Picard number of $X$.
\begin{enumerate}
\item
When $\delta_{f}=1$, there exists a constant $C>0$ such that
\[
h_{X}^{+}(f^{n}(P))\leq Cn^{2r}h_{X}^{+}(P)
\] 
for all $n\geq1$ and $P\in X( \overline{\mathbb Q})$.
\item
Assume that $\delta_{f}>1$. 
Then there exists a constant $C>0$ such that
\[
h_{X}^{+}(f^{n}(P)) \leq Cn^{r-1}\delta_{f}^{n}h_{X}^{+}(P)
\]
for all $n\geq1$ and $P\in X( \overline{\mathbb Q})$.
\end{enumerate}
\end{Thm}

\begin{Rem}\label{remark on mor case}
In \cite{ks2}, Kawaguchi and Silverman prove a similar inequality under the same assumption of 
Theorem \ref{main 1}.
Moreover, they prove that the arithmetic degree $\alpha_{f}(P)$ exists and
is equal to one of the eigenvalues of the linear map 
$f^{*} \colon  N^{1}(X)_{\mathbb R} \longrightarrow N^{1}(X)_{\mathbb R}$.
Thus for a surjective morphism $f$, the first three parts of Conjecture \ref{ksconj}
and the inequality $\alpha_{f}(P)\leq \delta_{f}$ follows..
\end{Rem}

\begin{Rem}
The exponent $2r$ in Theorem \ref{main 1} (1) is the best possible. 
For example, let $X$ be an elliptic curve with identity element $0\in X$ and $P\in X$ a non-torsion point.
Let $f=T_{P} \colon X \longrightarrow X$  be the translation by $P$.
Then, $\delta_{f}=1$ since $f^{*}={\rm id}$.
Let $h$ be the Neron-Tate height  on $X$.
Then $h^{+}(f^{n}(0))=h^{+}(nP)=\max\{1, n^{2}h(P)\}$.
\end{Rem}

If the Picard number of $X$ is one, we have the following stronger inequalities.

\begin{Thm}\label{main 4}
Let $X$ be a smooth projective variety of Picard number one.
Let $f \colon  X\dashrightarrow X$ be a dominant rational map.
\begin{enumerate}
\item
For a positive integer $k>0$, there exists a constant $C>0$ such that
\[
h_{X}^{+}(f^{n}(P))\leq Cn^{2}\rho((f^{k})^{*})^{n/k}h_{X}^{+}(P)
\]
for all $P\in X_{f}(\overline{ {\mathbb{Q}}})$ and $n\geq 1$.
\item
Let $k>0$ be a positive integer.
Assume that $\rho((f^{k})^{*})>1$.
Then there exists a constant $C>0$ such that
\[
h_{X}^{+}(f^{n}(P))\leq C\rho((f^{k})^{*})^{n/k}h_{X}^{+}(P)
\]
for all $P\in X_{f}(\overline{ {\mathbb{Q}}})$ and $n\geq 0$.
\end{enumerate}
\end{Thm}

A  dominant rational map $f$ is said to be algebraically stable if 
$(f^{n})^{*}=(f^{*})^{n} \colon  N^{1}(X)_{\mathbb R} \longrightarrow N^{1}(X)_{\mathbb R}$ for all $n>0$.
In this case, $\delta_{f}=\rho(f^{*})$.
As a corollary of Theorem \ref{main},
we get the following.

\begin{Prop}\label{main 5}
Assume that the Picard number of $X$ is one and let $f \colon X \dashrightarrow X$ be an algebraically stable dominant rational map
with $\delta_{f}>1$. Then the limit
\[
\hat{h}_{X,f}(P)=\lim_{n\to \infty}\frac{h_{X}(f^{n}(P))}{\delta_{f}^{n}}
\]
exists for all $P \in X_{f}( \overline{\mathbb Q})$.
\end{Prop}
More generally,
\begin{Prop}\label{main 6}
Let $X$ be a smooth projective variety over $ \overline{\mathbb Q}$.
Let $f \colon  X\dashrightarrow X$ be a dominant rational self-map defined over $ \overline{\mathbb Q}$.
Assume $\delta_{f}>1$ and there exists a nef $ {\mathbb{R}}$-divisor $H$ on $X$ such that $f^{*}H \equiv \delta_{f}H$.
Fix a height function $h_{H}$ associated with $H$.
Then for any $P \in X_{f}( \overline{\mathbb Q})$, 
the limit
\[
\hat{h}_{X,f}(P)=\lim_{n \to \infty} \frac{h_{H}(f^{n}(P))}{\delta_{f}^{n}}
\]
converges or diverges to $-\infty$.
\end{Prop}

\begin{que}
Are there any examples that the limits diverge to $-\infty$ ?
\end{que}

The function $\hat{h}_{X,f}$ is the function which is called the canonical height function in \cite{s}.
The canonical height functions of dynamical systems of self-morphisms are systematically studied in \cite{callsilv}.
On the other hand, little is known about the canonical heights of rational maps.
There are several recent studies on them.
In \cite[Theorem D]{jr}, it is proved that any birational self-maps of surfaces with dynamical degree greater than one
admit canonical heights up to birational conjugate.
In \cite{k2}, the canonical heights of regular affine automorphisms are studied in detail.

We prove Theorem \ref{main 1} in \S \ref{mor case},
Theorem \ref{main} in \S \ref{sec:FI},
Theorem \ref{main 4} and Proposition \ref{main 5}, \ref{main 6} in \S \ref{section:picard rank one case}.
In the proof of Theorem \ref{main 4}, we use the computation in
the proof of Theorem \ref{FI} in \S \ref{sec:FI}.

In this paper, we give a method to estimate $ h_{H}(f^{n}(P))$ in terms of the behavior 
of $f$ on the group $ N^{1}(X)_{\mathbb R}$
by controlling error terms arising from divisors numerically equivalent to zero.
We give an expression of error terms as a linear combinations of fixed height functions whose coefficients
can be controlled easily.

\begin{Rem}\label{remark on heights}
Let $D$ be an $ {\mathbb{R}}$-divisor on $X$.
Then $D$ determines a unique (logarithmic) Weil height function $h_{D}$ up to bounded functions as follows.
When $D$ is a very ample integral divisor, $h_{D}$ is the composite of the embedding by $|D|$ and the height on the projective space.
For general $D$, we write 

\begin{align}
\label{very ample sum}
D=\sum_{i=1}^{m}a_{i}H_{i} 
\end{align}
where $a_{i}$ are real numbers and $H_{i}$ are very ample divisors.
Then we define
\[
h_{D}=\sum_{i=1}^{m}a_{i}h_{H_{i}}.
\]
The function $h_{D}$ does not depend on the choice of the representation (\ref{very ample sum})
up to bounded function (see \cite{bg,hs,l} for the detail).
We call any representative of the class $h_{D} \mod (\text{bounded functions})$ a height function
associated with $D$.
We call a height function associated with an ample divisor an ample height function.

In the above definition, theorems and proposition, we fix a height function $h_{X}$.
Actually, for the definition of arithmetic degree, we can replace $h_{X}$ by any ample height functions.
Also, the above theorems and proposition are valid for all ample height functions $h_{X}$.
Indeed, note that for any ample height functions $h,h'$, there exists a positive number $c$ such that
\begin{align*}
ch^{+}\geq {h'}^{+},\ c{h'}^{+}\geq h^{+}
\end{align*}
on $X( \overline{\mathbb Q})$.
Thus, for the proof of the above theorems,
it is enough to prove them for a particular ample height function.
\end{Rem}

\begin{Rem}[Other ground fields]
All of the results and arguments in this paper remain valid without change for other ground fields $\overline{K}$
of characteristic $0$ where $K$ is a field with a set of non-trivial absolute values satisfying the product formula.
The main theorems (Theorem \ref{main}, \ref{main 1}) also hold over a field of positive characteristic, see Appendix \ref{pos chara}.
\end{Rem}

\begin{flushleft}
{\bf Notation.}
\end{flushleft}

\begin{notation}
\item[$||\ ||$] For a real vector $v\in {\mathbb{R}}^{n}$ or  a real matrix $M\in M_{n\times m}( {\mathbb{R}})$, $||v||$ and $||M||$ are
the maximum among the absolute values of the coordinates.
\item[$\equiv$] For two divisors $D_{1}, D_{2}$ on a projective variety, $D_{1}\equiv D_{2}$ means $D_{1}$ and $D_{2}$ are numerically equivalent.
\item[$\left<\  ,\ \right>$] For two column vectors $v=(v_{1}, \dots, v_{n}), w=(w_{1},\dots, w_{n})$ of the same size,  we write
$\left<v,w\right>=\sum v_{i}w_{i}$. We use this notation whenever the multiplication $v_{i}w_{i}$ is defined
(e.g. $v_{i}$ are real numbers and $w_{i}$ are $ {\mathbb{R}}$-divisors or real valued functions).
Similarly, for a real matrix $M$ and a vector $w$ entries in divisors or real valued functions, $Mw$ is defined in the obvious manner.
\item[${\bf h}\circ f$]  For a vector valued function ${\bf h}=(h_{1},\dots, h_{n})$ on a set $X$ and a map $f$ to $X$, we write
${\bf h}\circ f=(h_{1}\circ f,\dots, h_{n}\circ f)$.
 \end{notation}

\section{Endomorphism case}\label{mor case}

We first treat the case where $f$ is a morphism.
The purpose of this section is to prove the following theorem.

\begin{Thm}[Theorem \ref{main 1}]\label{bound for mor case}
Let X be a projective variety over ${\overline {\mathbb{Q}}}$ and
$f \colon X \longrightarrow X$ be a surjective morphism defined over $ \overline{\mathbb Q}$.
Let $\delta_{f}$ be the spectral radius of $f^{*} \colon N^{1}(X)_{ {\mathbb{R}}} \longrightarrow N^{1}(X)_{ {\mathbb{R}}}$.
(Actually, $\delta_{f}$ is equal to the dynamical degree of $f$ which is defined by taking a resolution of singularities.)
Let $r=\dim N^{1}(X)_{\mathbb R}$ be the Picard number of $X$.
Fix an ample height function $h_{X}$ on $X$.
\begin{enumerate}
\item
When $\delta_{f}=1$, there exists a constant $C>0$ such that
\[
h_{X}^{+}(f^{n}(P))\leq Cn^{2r}h_{X}^{+}(P)
\] 
for all $n\geq1$ and $P\in X( \overline{\mathbb Q})$.

\item
Assume that $\delta_{f}>1$. 
Then there exists a constant $C>0$ such that
\[
h_{X}^{+}(f^{n}(P)) \leq Cn^{r-1}\delta_{f}^{n}h_{X}^{+}(P)
\]
for all $n\geq1$ and $P\in X( \overline{\mathbb Q})$.
\end{enumerate}
\end{Thm}

\begin{proof}

Let $D_{1},\dots ,D_{r}$ be $ {\mathbb{R}}$-divisors
which form a basis for $ N^{1}(X)_{ {\mathbb{R}}}$.
Let $H$ be an ample divisor on $X$ such that $H +D_{i},\ H-D_{i}\ (i=1,\dots, r)$ are ample.
For $ {\mathbb{R}}$-divisors $\alpha, \beta$, $\alpha \equiv \beta$
means $\alpha$ and $\beta$ are numerically equivalent.
Let $f^{*}D_{i}\equiv \sum_{k=1}^{r}a_{ki}D_{k}$, and $A=(a_{ki})_{k,i}$.
We can write $H \equiv \sum_{i=1}^{r}c_{i}D_{i}$.
Then
\[
f^{*}H\equiv \sum_{j=1}^{r}\sum_{k=1}^{r}c_{j}a_{kj}D_{k}=
\left< A
\left( 
\begin{array}{c}
c_{1}\\
c_{2}\\
\vdots \\
c_{r}
\end{array}
\right),
\left(
\begin{array}{c}
D_{1}\\
D_{2}\\
\vdots \\
D_{r}
\end{array}
\right)
\right>
=\left<A\vec{c},\vec{D}\right>.
\]
Let
\begin{align}
\label{eq:1}
&E=f^{*}H-\left<A\vec{c},\vec{D}\right>\\
\label{eq:2}
&E_{i}=f^{*}D_{i}-\sum_{k=1}^{r}a_{ki}D_{k}.
\end{align}
Then
\[
\vec{E}=
\left(
\begin{array}{c}
E_{1}\\
E_{2}\\
\vdots \\
E_{r}
\end{array}
\right)
=f^{*} \vec{D}-{}^{ \rm t}A \vec{D}.
\]
Note that $E,E_{i}$ are numerically zero.

{\bf The choice of Height functions}.

First, we take and fix height functions $h_{D_{1}},\dots h_{D_{r}}$ associated with $D_{1},\dots , D_{r}$.
Next, we take and fix a height function $h_{H}$ associated with $H$ so that
$h_{H}\geq 1,\ h_{H}\geq |h_{D_{i}}|\ (i=1,\dots r)$.
Then $h_{D_{i}}\circ f,\ h_{H}\circ f$ are height functions associated with $f^{*}D_{i}$ and $f^{*}H$.
We write
\[
{\bf h}_{\vec{D}}=
\left(
\begin{array}{c}
h_{D_{1}}\\
h_{D_{2}}\\
\vdots \\
h_{D_{r}}
\end{array}
\right).
\]
We define
\begin{align}
\label{morhtE}
&h_{E}=h_{H}\circ f-
\left<A \vec{c}, {\bf h}_{\vec{D}}\right>\\
\label{morhtvecE}
&{\bf h}_{\vec{E}}=
\left(
\begin{array}{c}
h_{E_{1}}\\
h_{E_{2}}\\
\vdots \\
h_{E_{r}}
\end{array}
\right)=
{\bf h}_{\vec{D}}\circ f -{}^{\rm t}A {\bf h}_{\vec{D}}\ .
\end{align}
Then, by (\ref{eq:1})(\ref{eq:2}), $h_{E}$ and $h_{E_{i}}$ are height functions
associated with $E$ and $E_{i}$.
Now, since $E,E_{i}$ are numerically zero, there exists a constant $C>0$ such that for all
$Q\in X({\overline {\mathbb{Q}}})$
\begin{align}
\label{bound for e}
&|h_{E}(Q)|\leq C\sqrt[]{h_{H}(Q)}\\
\label{bound for e_{i}}
&|h_{E_{i}}(Q)|\leq C\sqrt[]{h_{H}(Q)}\ \ i=1,\dots ,r.
\end{align}
See for example \cite[Theorem B.5.9]{hs}  and Proposition \ref{root bound on proj var}.

Let us begin the estimation of $ h_{H}(f^{n}(P))$.
Let $P\in X({\overline {\mathbb{Q}}})$ be an arbitrary point.
Then we have
\[
h_{H}(f(P))=h_{E}(P)+\left<A \vec{c}, {\bf h}_{\vec{D}}\right>(P).
\]
For $n\geq2$, we have
\begin{align*}
h_{H}(f^{n}(P))=&(h_{H}\circ f)(f^{n-1}(P))-\left<A\vec{c}, {\bf h}_{\vec{D}}\right>(f^{n-1}(P))\\
&+\left<A\vec{c}, {\bf h}_{\vec{D}}\circ f \right>(f^{n-2}(P))-\left<A^{2}\vec{c}, {\bf h}_{\vec{D}}\right>(f^{n-2}(P))\\
&+\cdots \\
&+\left<A^{n-2}\vec{c}, {\bf h}_{\vec{D}}\circ f\right>(f(P))-\left<A^{n-1}\vec{c}, {\bf h}_{\vec{D}}\right>(f(P))\\
&+\left<A^{n-1}\vec{c}, {\bf h}_{\vec{D}}\circ f\right>(P)\\
=&h_{E}(f^{n-1}(P))\\
&+\left<A \vec{c}, {}^{\rm t}A {\bf h}_{\vec{D}}+ {\bf h}_{\vec{E}}\right>(f^{n-2}(P))-\left<A^{2}\vec{c}, {\bf h}_{\vec{D}}\right>(f^{n-2}(P))\\
&+\cdots \\
&+\left<A^{n-2} \vec{c}, {}^{\rm t}A {\bf h}_{\vec{D}}+ {\bf h}_{\vec{E}}\right>(f(P))-\left<A^{n-1}\vec{c}, {\bf h}_{\vec{D}}\right>(f(P))\\
&+\left<A^{n-1}\vec{c}, {}^{\rm t}A {\bf h}_{\vec{D}}+ {\bf h}_{\vec{E}}\right>(P) & \qquad \text{by (\ref{morhtE})(\ref{morhtvecE})}\\
=&h_{E}(f^{n-1}(P))\\
&+\left<A \vec{c}, {\bf h}_{\vec{E}}\right>(f^{n-2}(P))\\
&+\cdots \\
&+\left<A^{n-2} \vec{c}, {\bf h}_{\vec{E}}\right>(f(P))\\
&+\left<A^{n-1} \vec{c}, {\bf h}_{\vec{E}}\right>(P)+\left<A^{n} \vec{c}, {\bf h}_{\vec{D}}\right>(P).
\end{align*}
By (\ref{bound for e})(\ref{bound for e_{i}})
\[
|\left<A^{m} \vec{c}, {\bf h}_{\vec{E}}\right>(Q)|\leq r^{2}\| \vec{c}\|\|A^{m}\|C\sqrt[]{h_{H}(Q)}\ \ \text{for}\ Q\in X(\overline {\mathbb{Q}}).
\]
Also, by the choice of $h_{H}$ and $h_{D_{i}}$, we have
\[
|\left<A^{n} \vec{c}, {\bf h}_{\vec{D}}\right>(P)|\leq r^{2}\| \vec{c}\|\|A^{n}\|h_{H}(P).
\]
Thus
\begin{align}
\label{root ineq}
h_{H}(f^{n}(P))\leq &C\left(\sqrt[]{h_{H}(f^{n-1}(P))}+r^{2}\| \vec{c}\|\|A\|\sqrt[]{h_{H}(f^{n-2}(P))}+\cdots \right. \\
&\left.+r^{2}\| \vec{c}\|\|A^{n-2}\|\sqrt[]{h_{H}(f(P))}+r^{2}\| \vec{c}\|\|A^{n-1}\|\sqrt[]{h_{H}(P)}\right)+r^{2}\| \vec{c}\|\|A^{n}\|h_{H}(P). \notag
\end{align}

For simplicity, we write $\delta=\delta_{f}$.
Let $\rho(f^{*})$ be the spectral radius of the linear map
$f^{*} \colon  N^{1}(X)_{\mathbb R} \longrightarrow N^{1}(X)_{\mathbb R}$.
Let $\rho(A)$ be the spectral radius of the matrix $A$.
By definition, we have $\delta=\rho(f^{*})=\rho(A)=\lim_{n\to \infty}\|A^{n}\|^{1/n}$.
Note that
\[
\frac{r^{2}\| \vec{c}\|\| A^{k}\|}{k^{r-1}\rho(A)^{k}}=\frac{r^{2}\| \vec{c}\|\| A^{k}\|}{k^{r-1}\delta^{k}}
\]
is bounded with respect to $k>0$.

Let $C_{1}=\sup_{k>0}\left\{ r^{2}\| \vec{c}\|\| A^{k}\|\big/k^{r-1}\delta^{k}\right\}$.
Set $C_{2}=\max\left\{1,C_{1}, CC_{1}, C\right\}$. 
Then dividing inequality (\ref{root ineq}) by $n^{r-1}\delta^{n}$, we get
\begin{align}
\label{root ineq 2}
&\frac{ h_{H}(f^{n}(P))}{n^{r-1}\delta^{n}} \\
&\leq C\left( \frac{r^{2}\| \vec{c}\|\|A^{n-1}\|}{n^{r-1}\delta^{n}} \sqrt[]{h_{H}(P)}+\right.\notag\\
&\left.\ \ \ \sum_{k=1}^{n-2} \frac{r^{2}\| \vec{c}\|\|A^{n-1-k}\|}{(n-1-k)^{r}\delta^{n-1-k}} \sqrt[]{\frac{ h_{H}(f^{k}(P))}{k^{r-1}\delta^{k}}}
\frac{(n-1-k)^{r-1}k^{(r-1)/2}}{n^{r-1}\delta^{1+k/2}} \right.\notag\\
&\left.\ \ \ +  \sqrt[]{\frac{ h_{H}(f^{n-1}(P))}{(n-1)^{r-1}\delta^{n-1}}} \frac{(n-1)^{(r-1)/2}}{n^{r-1}\delta^{1+(n-1)/2}} \right)
+\frac{r^{2}\| \vec{c}\|\|A^{n}\|}{n^{r-1}\delta^{n}}h_{H}(P)\notag\\
&\leq  C_{2}\left( \sqrt[]{h_{H}(P)} +\sum_{k=1}^{n-2} \sqrt[]{\frac{ h_{H}(f^{k}(P))}{k^{r-1}\delta^{k}}}
\frac{(n-1-k)^{r-1}k^{(r-1)/2}}{n^{r-1}\delta^{1+k/2}}\right.\notag\\
&\left.\ \ \ +\sqrt[]{\frac{ h_{H}(f^{n-1}(P))}{(n-1)^{r-1}\delta^{n-1}}} \frac{(n-1)^{(r-1)/2}}{n^{r-1}\delta^{1+(n-1)/2}}+h_{H}(P)\right).\notag
\end{align}

First we assume that $\delta>1$.
Then $k^{(r-1)/2}\big/\delta^{1+k/2}$ is bounded with respect to $k$.
Thus, there exists a constant $C_{3}>0$ which is independent of $n, P$ so that
\[
\frac{ h_{H}(f^{n}(P))}{n^{r-1}\delta^{n}} \leq 
C_{3}\left( \sqrt[]{h_{H}(P)}+ \sum_{k=1}^{n-1}\sqrt[]{\frac{ h_{H}(f^{k}(P))}{k^{r-1}\delta^{k}}}+h_{H}(P)\right).
\]
Applying Lemma \ref{seq lem} to the sequence
$a_{0}=h_{H}(P), a_{n}=h_{H}(f^{n}(P))\big/n^{r}\delta^{n}\ (n\geq1)$, there exists a constant $C_{4}>0$ independent of $n, P$ such that
\[
\frac{ h_{H}(f^{n}(P))}{n^{r-1}\delta^{n}} \leq C_{4}n^{2}h_{H}(P)
\]
for all $n\geq1$.
Again from (\ref{root ineq 2}),
\[
\frac{ h_{H}(f^{n}(P))}{n^{r-1}\delta^{n}}\leq
C_{2}\left( \sqrt[]{h_{H}(P)}+\sum_{k=1}^{n-1} \sqrt[]{C_{4}h_{H}(P)} \frac{k^{1+(r-1)/2}}{\delta^{1+k/2}}+h_{H}(P)\right).
\]
Since $\sum_{k=1}^{\infty}k^{1+(r-1)/2}\big/\delta^{1+k/2}$ is convergent, there exists a constant
$C_{5}>0$ independent of $n, P$ such that
\[
\frac{ h_{H}(f^{n}(P))}{n^{r-1}\delta^{n}}\leq C_{5}h_{H}(P).
\]
Thus $h_{H}(f^{n}(P))\leq C_{5}n^{r-1}\delta^{n}h_{H}(P)$.
Now, since $h_{H}$ and $h_{X}$ are ample height functions and we take $h_{H}\geq 1$,
there exists an integer $m >0$ such that
\[
mh_{H}\geq h_{X}^{+},\ mh_{X}^{+}\geq h_{H}.
\]
Thus
\[
h_{X}^{+}(f^{n}(P))\leq m h_{H}(f^{n}(P)) \leq mC_{5}n^{r-1}\delta^{n}h_{H}(P)
\leq m^{2}C_{5}n^{r-1}\delta^{n}h_{X}^{+}(P).
\]
This completes the proof of Theorem \ref{bound for mor case}(2).

Now assume that $\delta=1$.
Dividing both sides of (\ref{root ineq 2}) by $n^{r-1}$, we get
\begin{align*}
\frac{ h_{H}(f^{n}(P))}{n^{2r-2}}
\leq  C_{2}&\left( \frac{\sqrt[]{h_{H}(P)}}{n^{r-1}} +\sum_{k=1}^{n-2} \sqrt[]{\frac{ h_{H}(f^{k}(P))}{k^{2r-2}}}
\frac{(n-1-k)^{r-1}k^{r-1}}{n^{2r-2}}\right.\\
&\ \ \left.+\sqrt[]{\frac{ h_{H}(f^{n-1}(P))}{(n-1)^{2r-2}}} \frac{(n-1)^{r-1}}{n^{2r-2}}+\frac{h_{H}(P)}{n^{r-1}}\right)\\
\leq C_{2}&\left( \sqrt[]{h_{H}(P)} +\sum_{k=1}^{n-1} \sqrt[]{\frac{ h_{H}(f^{k}(P))}{k^{2r-2}}}+h_{H}(P)\right).
\end{align*}
By Lemma \ref{seq lem}, there exists a constant $C_{6}>0$ independent of $n, P$ such that
\[
h_{H}(f^{n}(P)) \leq C_{6}n^{2r}h_{H}(P) \ \ \ \text{for all $n\geq1$}.
\]
By the same argument at the end of the proof of (2), this proves Theorem \ref{bound for mor case}(1).
\end{proof}

\section{Rational self-map case}\label{sec:FI}

Now we prove the main theorem of this paper.

\begin{Thm}[Theorem \ref{main}]\label{MI}
Let $X$ be a smooth projective variety over $ \overline{\mathbb Q}$ and $f \colon X \dashrightarrow X$
be a dominant rational map defined over $ \overline{\mathbb Q}$.
Let $\delta_{f}$ be the first dynamical degree of $f$.
Fix an ample height function $h_{X}$ on $X$.
Then, for any $\epsilon>0$, there exists $C>0$ such that
\[
h_{X}^{+}(f^{n}(P))\leq C(\delta_{f}+\epsilon)^{n}h_{X}^{+}(P)
\]
for all $n\geq0$ and $P \in X_{f}( \overline{\mathbb Q})$.
In particular, for any $P\in X_{f}( \overline{\mathbb Q})$, we have
\[
\overline{\alpha}_{f}(P) \leq \delta_{f}.
\]
\end{Thm}

We deduce this theorem from the following theorem.

\begin{Thm}\label{FI}
Let $X$ be a smooth projective variety over $ \overline{\mathbb Q}$ and $f \colon X \dashrightarrow X$
be a dominant rational map defined over $ \overline{\mathbb Q}$ with first dynamical degree $\delta_{f}$.
Fix an ample height function $h_{X}$ on $X$.
Then, for any $\epsilon>0$, there exist a positive integer $k$ and a constant $C>0$ such that
\[
h_{X}^{+}(f^{nk}(P))\leq C(\delta_{f}+\epsilon)^{nk}h_{X}^{+}(P)
\]
for all $n\geq0$ and $P \in X_{f}( \overline{\mathbb Q})$.
\end{Thm}

\begin{Lem}\label{uniform bound lem}
In the situation of Theorem \ref{FI}, there exists a constant $C_{0}\geq 1$ such that
\[
h_{X}^{+}(f^{n}(P)) \leq C_{0}^{n}h_{X}^{+}(P)
\]
for all $n\geq0$ and $P \in X_{f}( \overline{\mathbb Q})$.
\end{Lem}
\begin{proof}
Let $H$ be an ample divisor on $X$.
Take a height function $h_{H}$ associated with $H$ so that $h_{H}\geq 1$.
Let $h_{f^{*}H}$ be a height function associated with $f^{*}H$.
Then, from \cite[Proposition 21]{ks}
\[
h_{H}(f(P)) \leq h_{f^{*}H}(P)+O(1)
\]
for all $P \in X_{f}( \overline{\mathbb Q})$.
Here $O(1)$ is a bounded function on $X_{f}( \overline{\mathbb Q})$ which depends on
$f, H, f^{*}H, h_{H}, h_{f^{*}H}$ but is independent of $P$.
Since $H$ is ample and $h_{H}\geq1$, for a sufficiently large $C_{0}\geq1$, we have
\[
h_{f^{*}H}(P)+O(1) \leq C_{0}h_{H}(P)
\]
for all $P \in X_{f}( \overline{\mathbb Q})$.
Thus, we get
\[
h_{H}(f(P)) \leq C_{0}h_{H}(P)
\]
for all $P \in X_{f}( \overline{\mathbb Q})$.
Therefore
\[
h_{H}(f^{n}(P)) \leq C_{0}^{n}h_{H}(P).
\]
By Remark \ref{remark on heights} or the same argument at the end of the proof of Theorem \ref{bound for mor case}(2), this proves
the statement.
\end{proof}

\begin{proof}[Proof of Theorem \ref{FI} $\Longrightarrow$ Theorem \ref{MI}]
From Theorem \ref{FI}, for any $\epsilon>0$, there exist a positive integer $k$ and a positive constant
$C>0$ such that
\[
h_{X}^{+}(f^{nk}(P))\leq C(\delta_{f}+\epsilon)^{nk}h_{X}^{+}(P)
\]
for all $n\geq0$ and $P \in X_{f}( \overline{\mathbb Q})$.
For any integer $m\geq0$, we write $m=qk+t\ q\geq0, 0\leq t <k$.
Let $C_{0}$ be the constant in Lemma \ref{uniform bound lem}.
Then for any $P \in X_{f}( \overline{\mathbb Q})$,
\begin{align*}
h_{X}^{+}(f^{m}(P)) &\leq C(\delta_{f}+\epsilon)^{qk}h_{X}^{+}(f^{t}(P))\\
&\leq CC_{0}^{t}(\delta_{f}+\epsilon)^{qk}h_{X}^{+}(P)\\
&\leq CC_{0}^{k-1}(\delta_{f}+\epsilon)^{m}h_{X}^{+}(P).
\end{align*}
This proves the first statement in Theorem \ref{MI}.

The second statement is an easy consequence of the first one.
That is,
\begin{align*}
\overline{\alpha}_{f}(P) &= \limsup_{n\to \infty}h_{X}^{+}(f^{n}(P))^{1/n}\\
&\leq \limsup_{n\to \infty}\left(Ch_{X}^{+}(P)\right)^{1/n}(\delta_{f}+\epsilon)\\
&= \delta_{f}+\epsilon.
\end{align*}
Since $\epsilon$ is arbitrary, we get $ \overline{\alpha}_{f}(P) \leq \delta_{f}$.
\end{proof}

Before starting the proof of Theorem \ref{FI}, we prove an interesting corollary.

\begin{Cor}\label{divisible}
In the situation of Theorem \ref{FI}, 
\[
\overline{\alpha}_{f}(P)=\limsup_{n\to \infty}h_{X}^{+}(f^{nk}(P))^{1/nk}=\overline{\alpha}_{f^{k}}(P)^{1/k}
\]
for any $k>0$ and any point $P\in X_{f}(\overline{ {\mathbb{Q}}})$.
\end{Cor}
\begin{proof}
We compute
\begin{align*}
\overline{\alpha}_{f}(P)&=\limsup_{m\to \infty}h_{X}^{+}(f^{m}(P))^{1/m}\\
&=\limsup_{n\to \infty}\max_{0\leq i <k}h_{X}^{+}(f^{nk+i}(P))^{1/nk+i}\\
&\leq \limsup_{n\to \infty}\max_{0\leq i<k}(C_{0}^{i}h_{X}^{+}(f^{nk}(P)))^{1/nk+i} & \qquad \text{by Lemma \ref{uniform bound lem}}\\
&\leq \limsup_{n\to \infty} (C_{0}^{k-1}h_{X}^{+}(f^{nk}(P)))^{1/nk}\\
&=\limsup_{n\to \infty}h_{X}^{+}(f^{nk}(P))^{1/nk}\\
&\leq \overline{\alpha}_{f}(P).
\end{align*}
Then we have $ \overline{\alpha}_{f}(P)=\limsup_{n\to \infty}h_{X}^{+}(f^{nk}(P))^{1/nk}= \overline{\alpha}_{f^{k}}(P)^{1/k}$.
\end{proof}

Now we turn to the proof of Theorem \ref{FI}.

\begin{proof}[Proof of Theorem \ref{FI}]

Let $D_{1},\dots , D_{r}$ be very ample divisors on $X$ which forms a basis for $ N^{1}(X)_{\mathbb R}$.
Take an ample divisor $H$ on $X$ so that $H\pm D_{i},\ i=1,\dots,r$ are ample and 
if we write $H \equiv \sum_{i=1}^{r}c_{i}D_{i}$ then $c_{i}\geq 0$.

We take a resolution of indeterminacy $p \colon Y \longrightarrow X$ of $f$ as follows.
$p$ is a sequence of blowing ups at smooth centers and the images of centers in $X$
are contained in the indeterminacy locus $I_{f}$ of $f$.
Let $g=f\circ p$.
\[
\xymatrix{
&Y \ar[ld]_{p} \ar[rd]^{g}&\\
X \ar@{-->}[rr]_{f}&&X
}
\]
Let $ {\rm Exc}(p)$ be the exceptional locus of $p$.
By the negativity lemma (see for example \cite[Lemma 3.39]{km}),
\[
Z_{i}=p^{*}{p}_{*}g^{*}D_{i}-g^{*}D_{i}
\]
is an effective divisor on $Y$ whose support is contained in $ {\rm Exc}(p)$.
Let $F_{i}=g^{*}D_{i}$ for $i=1,\dots,r$.
Then,
\begin{align}
\label{Z_{i}}
p^{*}{p}_{*}F_{i}-F_{i} = Z_{i}.
\end{align}
Take divisors $F_{r+1},\dots,F_{s}$ on $Y$ so that $F_{1},\dots ,F_{s}$
forms a basis for $ N^{1}(Y)_{\mathbb R}$.
There exists an ample $ {\mathbb{Q}}$-divisor ${H'}$ on $Y$ such that $p^{*}H-{H'}$
is an effective $ {\mathbb{Q}}$-divisor whose support is contained in $ {\rm Exc}(p)$.
Indeed, take an effective $p$-exceptional divisor $G$ such that $-G$ is $p$-ample.
(For the existence of such a divisor, see for example \cite[Lemma 2.62]{km}).
Then, for sufficiently large $N>0$, $H'=-\frac{1}{N}G+p^{*}H$ satisfies desired properties.
Let
\begin{align}
\label{a_{mi}}
&g^{*}D_{i} \equiv \sum_{m=1}^{s}a_{mi}F_{m}\ \ (i=1,\dots r)\\
\label{b_{lj}}
&{p}_{*}F_{j} \equiv \sum_{l=1}^{r}b_{lj}D_{l} \ \ (j=1,\dots, s)
\end{align}
and
\begin{align*}
&A=(a_{mi})_{mi}\ \ \ s\times r\text{-matrix}\\
&B=(b_{lj})_{lj}\ \ r\times s\text{-matrix}.
\end{align*}
By the definition of $F_{j}$, $A$ is the following form.
\begin{align}
\label{form of A}
A=\left(
\begin{array}{ccc}
1& &\\
&\ddots& \\
&&1\\
&&
\end{array}
\right).
\end{align}
Note that $BA$ is the representation matrix of $f^{*}$
with respect to the basis $D_{1},\dots,D_{r}$.
We write
\[
\vec{D}=\left( 
\begin{array}{c}
D_{1}\\
D_{2}\\
\vdots \\
D_{r}
\end{array}
\right), 
\vec{F}=\left( 
\begin{array}{c}
F_{1}\\
F_{2}\\
\vdots \\
F_{s}
\end{array}
\right),
\vec{c}=\left( 
\begin{array}{c}
c_{1}\\
c_{2}\\
\vdots \\
c_{r}
\end{array}
\right),
\vec{Z}=\left( 
\begin{array}{c}
Z_{1}\\
Z_{2}\\
\vdots \\
Z_{r}
\end{array}
\right).
\]
Let
\begin{align}
\label{E}
&E=g^{*}H-\left<A \vec{c}, \vec{F}\right>\\
\label{vecE'}
& \vec{{E'}}=\left( 
\begin{array}{c}
E_{1}'\\
E_{2}'\\
\vdots \\
E_{s}'
\end{array}
\right)
={p}_{*} \vec{F}-{}^{\rm t}B \vec{D}.
\end{align}
These are numerically zero divisors.

{\bf The choice of height functions}.

Fix height functions $h_{D_{1}},\dots,h_{D_{r}}$ associated with $D_{1},\dots,D_{r}$.
Fix a height function $h_{H}$ associated with $H$ so that
$h_{H}\geq1$ and $h_{H}\geq |h_{D_{i}}|$ for $i=1,\dots,r$.
Note that $h_{D_{1}},\dots,h_{D_{r}}$ and $h_{H}$ are independent of $f$.
 
We define $h_{F_{j}}=h_{D_{j}}\circ g,\ j=1,\dots,r$.
These are height functions associated with $F_{j}$.
For $j=r+1,\dots,s$, fix any height functions $h_{F_{j}}$ associated with $F_{j}$.
Fix height functions $h_{p_{*}F_{j}}$ associated with $p_{*}F_{j}$ for $j=1,\dots,s$.
We write
\begin{align*}
&{\bf h}_{\vec{D}}=\left(
\begin{array}{c}
h_{D_{1}}\\
h_{D_{2}}\\
\vdots \\
h_{D_{r}}
\end{array}
\right),\ 
{\bf h}_{\vec{F}}=\left(
\begin{array}{c}
h_{F_{1}}\\
h_{F_{2}}\\
\vdots \\
h_{F_{s}}
\end{array}
\right),\ 
{\bf h}_{{p}_{*}\vec{F}}=\left(
\begin{array}{c}
h_{{p}_{*}F_{1}}\\
h_{{p}_{*}F_{2}}\\
\vdots \\
h_{{p}_{*}F_{s}}
\end{array}
\right)\ .
\end{align*}
Define
\begin{align}
\label{htE'}
&{\bf h}_{\vec{{E'}}}=
\left(
\begin{array}{c}
h_{E_{1}'}\\
h_{E_{2}'}\\
\vdots \\
h_{E_{s}'}
\end{array}
\right)
={\bf h}_{{p}_{*}\vec{F}}-{}^{\rm t}B {\bf h}_{\vec{D}}\\
\label{htE}
&h_{E}=h_{H}\circ g-\left<A \vec{c}, {\bf h}_{\vec{F}}\right>\\
\label{htZ}
& {\bf h}_{\vec{Z}}=
\left(
\begin{array}{c}
h_{Z_{1}}\\
h_{Z_{2}}\\
\vdots \\
h_{Z_{r}}
\end{array}
\right)
=\left(
\begin{array}{c}
h_{{p}_{*}F_{1}}\\
h_{{p}_{*}F_{2}}\\
\vdots \\
h_{{p}_{*}F_{r}}
\end{array}
\right)
\circ p-
\left(
\begin{array}{c}
h_{F_{1}}\\
h_{F_{2}}\\
\vdots \\
h_{F_{r}}
\end{array}
\right)\ .
\end{align}

By (\ref{vecE'}), (\ref{E}) and (\ref{Z_{i}}), 
$h_{E_{j}'}$ is a height function associated with $E_{j}'$ for $j=1,\dots, s$, $h_{E}$ is the one with $E$
and $h_{Z_{i}}$ is the one with $Z_{i}$ for $i=1,\dots,r$.
By adding a bounded function to $h_{{p}_{*}F_{i}}$, we may assume that $h_{Z_{i}} \geq 0$
on $Y\setminus Z_{i}$
(see for example \cite[Theorem B.3.2(e)]{hs}).
Fix a height function $h_{{H'}}\geq1$ associated with ${H'}$.
Fix a height function $h_{p^{*}H-{H'}}$ associated with $p^{*}H-{H'}$
so that $h_{p^{*}H-{H'}}\geq 0$ on $Y\setminus {\rm Exc}(p)$.
Note that there exists a constant $\gamma \geq0$ such that
\begin{align}
\label{leqgamma}
h_{H}\circ p \geq h_{p^{*}H-{H'}}+h_{{H'}}-\gamma\ \ \ \text{on $Y( \overline{\mathbb Q})$}. 
\end{align}

Since $E, {E'_{j}}$ are numerically zero, there exists a constant $C>0$ such that
\begin{align}
\label{rootboundE}
&|h_{E}|\leq C \sqrt[]{h_{{H'}}}\\
\label{rootboundE'}
&|h_{{E'_{j}}}|\leq C \sqrt[]{h_{H}}.
\end{align}

Let $M(f)$ be the representation matrix of the linear map
$f^{*} \colon  N^{1}(X)_{\mathbb R} \longrightarrow N^{1}(X)_{\mathbb R}$
with respect to the basis $D_{1},\dots,D_{r}$.

\begin{Claim}
Let $R= \max\{1, r^{2}\| \vec{c}\|\|M(f)\|\}$.
Then there exists $K>0$ such that
\[
h_{H}(f^{n}(P)) \leq Kn^{2}R^{n}h_{H}(P)
\]
for all $n\geq1$ and $P\in X_{f}( \overline{\mathbb Q})$.
Note that the constant $K$ depends on $f$ but
$h_{H}, r, \vec{c}$ and $D_{1},\dots,D_{r}$ do not depend on $f$.
\end{Claim}
\begin{proof}[Proof of the claim]

Let $P \in X_{f}( \overline{\mathbb Q})$.
Note that $p^{-1}$ is defined at $f^{i}(P)$ for every $i\geq0$.
For $n\geq1$
\begin{alignat}{2}
\label{funddef}
&h_{H}(f^{n}(P))\\
=&(h_{H}\circ g)(p^{-1}f^{n-1}(P))
-\left<A \vec{c}, {\bf h}_{{p}_{*}\vec{ F}}\circ p\right>(p^{-1}f^{n-1}(P))
+\left<A \vec{c}, {\bf h}_{{p}_{*}\vec{F}}\right>(f^{n-1}(P))\notag\displaybreak[1]\\
     \intertext{by (\ref{htE'})(\ref{htE}),}
=& \left<A \vec{c}, {\bf h}_{\vec{F}}-{\bf h}_{{p}_{*}\vec{ F}}\circ p\right>(p^{-1}f^{n-1}(P))
+h_{E}(p^{-1}f^{n-1}(P))
+\left<BA \vec{c}, {\bf h}_{\vec{D}}\right>(f^{n-1}(P)) \notag\displaybreak[0]\\
&+\left<A \vec{c}, {\bf h}_{\vec{{E'}}}\right>(f^{n-1}(P))                    \notag\displaybreak[3]\\
     \intertext{by (\ref{htZ}),}
=& \left<\vec{c}, - {\bf h}_{\vec{Z}}\right>(p^{-1}f^{n-1}(P))
+h_{E}(p^{-1}f^{n-1}(P))
+\left<BA \vec{c}, {\bf h}_{\vec{D}}\right>(f^{n-1}(P)) 
+\left< \vec{c}, {}^{\rm t}A{\bf h}_{\vec{{E'}}}\right>(f^{n-1}(P))  \notag\displaybreak[3]\\
 \intertext{since $h_{Z_{i}}\geq0$ on $Y\setminus {\rm Exc}(p)$,}
\leq & h_{E}(p^{-1}f^{n-1}(P))
+\left<BA \vec{c}, {\bf h}_{\vec{D}}\right>(f^{n-1}(P))
+\left< \vec{c}, {}^{\rm t}A{\bf h}_{\vec{{E'}}}\right>(f^{n-1}(P))  \notag\displaybreak[3]\\
\intertext{by (\ref{form of A})(\ref{rootboundE})(\ref{rootboundE'}),} 
\leq & r^{2}\| \vec{c}\|\|BA\|h_{H}(f^{n-1}(P))
+r\| \vec{c}\|C \sqrt[]{h_{H}(f^{n-1}(P))}
+C \sqrt[]{h_{{H'}}(p^{-1}(f^{n-1}(P)))}     \notag\displaybreak[3]\\
\intertext{by (\ref{leqgamma}) and $h_{p^{*}H-{H'}}\geq0$ on $Y\setminus {\rm Exc}(p)$,}
\leq & r^{2}\| \vec{c}\|\|BA\|h_{H}(f^{n-1}(P))
+r\| \vec{c}\|C \sqrt[]{h_{H}(f^{n-1}(P))}
+C \sqrt[]{h_{H}(f^{n-1}(P))+\gamma}   \notag\displaybreak[0].
\end{alignat}
Note that $C, \gamma$ depend on $f$.
On the other hand, $r, H, D_{1},\dots, D_{r}$, and $h_{H}$ do not depend on $f$.
Thus $ \vec{c}$ also does not depend on $f$.

Since $BA$ is the representation matrix of $f^{*}$ with respect to $D_{1},\dots,D_{r}$, 
$BA=M(f)$ and $R=\max\{ 1, r^{2}\| \vec{c}\|\|BA\|\}$.
Then, dividing the both sides of (\ref{funddef}) by $R^{n}$, we get
\begin{align*}
\frac{h_{H}(f^{n}(P))}{R^{n}}\leq & \frac{h_{H}(f^{n-1}(P))}{R^{n-1}}\\
&+r\| \vec{c}\|C \sqrt[]{\frac{h_{H}(f^{n-1}(P))}{R^{n-1}}}
+C \sqrt[]{\frac{h_{H}(f^{n-1}(P))}{R^{n-1}}+\gamma}\ .
\end{align*}
Let
\[
a_{n}=\frac{ h_{H}(f^{n}(P))}{R^{n}}\ \ \ \text{for $n\geq0$}.
\]
Then $a_{n}>0$ and $a_{0}=h_{H}(P)$ and the sequence $(a_{n})_{n}$ satisfies the following
inequality.
\[
a_{n} \leq a_{n-1}+r\| \vec{c}\|C \sqrt[]{a_{n-1}}+C \sqrt[]{a_{n-1}+\gamma}
\]
By Lemma \ref{seq lem0}, there exist a constant $K>0$ independent of $n, P$ such that
\[
a_{n}\leq Kn^{2}a_{0}\ \ \ \text{for all $n\geq1$}.
\]
Therefore
\begin{align*}
h_{H}(f^{n}(P))&\leq Kn^{2}R^{n}h_{H}(P).
\end{align*}
Thus we get the claim.
 \end{proof}

Now, fix any positive real number $\epsilon>0$.
Let $\delta=\delta_{f}$.
Let $M(f^{k})$ be the representation matrix of $(f^{k})^{*} \colon  N^{1}(X)_{\mathbb R} \longrightarrow N^{1}(X)_{\mathbb R}$
with respect to the basis $D_{1},\dots , D_{r}$.
Since $\lim_{k\to \infty}\|M(f^{k})\|^{1/k}=\delta$,
there exists a positive integer $k>0$ such that
\begin{align}
\label{ep'}
\frac{\|M(f^{k})\|}{(\delta+\epsilon)^{k}}r^{2}\| \vec{c}\| <1.
\end{align}
Fix such a $k$. We apply the claim to $f^{k}$ in the place of $f$.
Then,
\begin{align*}
h_{H}(f^{kn}(P))\leq
Kn^{2}\left(\frac{R}{(\delta+\epsilon)^{k}}\right)^{n}(\delta+\epsilon)^{kn}h_{H}(P).
\end{align*}
Recall $R=\max\{ 1, r^{2}\| \vec{c}\|\|M(f^{k})\|\}$.
Thus, by (\ref{ep'}) 
\[
\frac{R}{(\delta+\epsilon)^{k}}<1.
\]
Thus there exists a constant $K'$ such that
\[
Kn^{2}\left(\frac{R}{(\delta+\epsilon)^{k}}\right)^{n}\leq K'
\]
for all $n$.
Then we get
\[
h_{H}(f^{kn}(P))\leq
K'(\delta+\epsilon)^{kn}h_{H}(P).
\]
By Remark \ref{remark on heights} or the same argument at the end of the proof of Theorem \ref{bound for mor case}(2), this proves
Theorem \ref{FI}(2).

\end{proof}

\begin{Rem}
One can prove Theorem \ref{MI} over any ground field $K$ such that Weil height functions can be defined.
If the characteristic of $K$ is zero, the same proof works.
For the case when the characteristic of $K$ is positive, see Appendix \ref{pos chara}.
\end{Rem}

\section{Picard rank one case}\label{section:picard rank one case}

When the Picard number of $X$ is one, we can say much more about the behavior of the sequence $\{h_{X}(f^{n}(P))\}_{n}$.

\begin{Thm}[Theorem \ref{main 4}]\label{picone}
Let $X$ be a smooth projective variety over $ \overline{\mathbb Q}$
of Picard number one.
Let $f \colon  X\dashrightarrow X$ be a dominant rational self-map defined over $ \overline{\mathbb Q}$.
Fix an ample height function $h_{X}$ on $X$.
\begin{enumerate}
\item
For any positive integer $k>0$, there exists a constant $C>0$ such that
\[
h_{X}^{+}(f^{n}(P))\leq Cn^{2}\rho((f^{k})^{*})^{n/k}h_{X}^{+}(P)
\]
for all $P\in X_{f}(\overline{ {\mathbb{Q}}})$ and $n\geq 1$.
\item
Let $k>0$ be a positive integer.
Assume that $\rho((f^{k})^{*})>1$.
Then there exists a constant $C>0$ such that
\[
h_{X}^{+}(f^{n}(P))\leq C\rho((f^{k})^{*})^{n/k}h_{X}^{+}(P)
\]
for all $P\in X_{f}(\overline{ {\mathbb{Q}}})$ and $n\geq 0$.
\end{enumerate}
\end{Thm}
\begin{proof}
We use the notation in the proof of Theorem \ref{FI}.
For simplicity, we write $\rho_{k}=\rho((f^{k})^{*})$ for $k>0$.
We apply (\ref{funddef}) to $f^{k}$.
By the assumption $r=1$, thus $BA=\rho_{k}$ is a real number.
By (\ref{funddef}),
\begin{align}
\label{piconeexp}
h_{H}(f^{nk}(P))=&-c_{1}h_{Z_{1}}(p^{-1}f^{k(n-1)}(P))+h_{E}(p^{-1}f^{k(n-1)}(P))\\
&+\rho_{k}c_{1}h_{D_{1}}(f^{k(n-1)}(P))+c_{1}h_{{E_{1}'}}(f^{k(n-1)}(P))\notag\\
\leq& \rho_{k}c_{1}h_{D_{1}}(f^{k(n-1)}(P))+C \sqrt[]{ h_{H}(f^{k(n-1)}(P))+\gamma}\notag\\
&+c_{1}C \sqrt[]{ h_{H}(f^{k(n-1)}(P))}\notag
\end{align}
Let $N=c_{1}D_{1}-H$. By the definition of $c_{1}$, this is a numerically zero divisor.
Define
\[
h_{N}=c_{1}h_{D_{1}}-h_{H}.
\]
Then, this is a height function associated with $N$.
Thus there exists a constant $\widetilde{C}>0$ such that
\[
|h_{N}| \leq \widetilde{C} \sqrt[]{h_{H}}.
\]
Then
\begin{align*}
h_{H}(f^{nk}(P))\leq& \rho_{k}h_{H}(f^{k(n-1)}(P))+\widetilde{C} \sqrt[]{ h_{H}(f^{k(n-1)}(P))}\\
&+C \sqrt[]{ h_{H}(f^{k(n-1)}(P))+\gamma}+c_{1}C \sqrt[]{ h_{H}(f^{k(n-1)}(P))}.
\end{align*}
Divide both sides of this inequality by $\rho_{k}^{n}$.
By Lemma \ref{seq lem0}, there exists a constant $\widetilde{K}>0$ (which is independent of $n,P$, but depends on $k$)
such that
\begin{align}
\label{pic one first}
h_{H}(f^{nk}(P))\leq \widetilde{K}n^{2}\rho_{k}^{nk/k}h_{H}(P)\ \ \ \text{for all $n\geq1$}.
\end{align}
By the same argument as in (Proof of Theorem \ref{FI} $\Longrightarrow $Theorem \ref{MI}),
we can prove the first statement.

Now assume $\rho_{k}>1$.
Then
\begin{align*}
\frac{ h_{H}(f^{nk}(P))}{\rho_{k}^{n}}\leq \frac{ h_{H}(f^{k(n-1)}(P))}{\rho_{k}^{n-1}}
+\left(\widetilde{C}+C+c_{1}C\right) \frac{\sqrt[]{ h_{H}(f^{k(n-1)}(P))}}{\rho_{k}^{n}}
+\frac{C \sqrt[]{\gamma}}{\rho_{k}^{n}}
\end{align*}

By (\ref{pic one first}),
\[
\sqrt[]{h_{H}(f^{k(n-1)}(P))} \leq \sqrt[]{ \widetilde{K}h_{H}(P)}(n-1)\rho_{k}^{(n-1)/2}
\]
and thus
\begin{align*}
&\sum_{n=1}^{\infty}\left\{ \left(\widetilde{C}+C+c_{1}C\right) \frac{\sqrt[]{ h_{H}(f^{k(n-1)}(P))}}{\rho_{k}^{n}}
+\frac{C \sqrt[]{\gamma}}{\rho_{k}^{n}} \right\}\\
\leq &\sum_{n=1}^{\infty}\left\{ \left(\widetilde{C}+C+c_{1}C\right) 
\frac{\sqrt[]{ \widetilde{K}h_{H}(P)}(n-1)\rho_{k}^{(n-1)/2}}{\rho_{k}^{n}}
+\frac{C \sqrt[]{\gamma}}{\rho_{k}^{n}} \right\}.
\end{align*}
Since $\rho_{k}>1$, there exists a constant $\widetilde{K}_{1}$ (independent of $n,P$)
such that
\[
\frac{h_{H}(f^{nk}(P))}{\rho_{k}^{n}}\leq \widetilde{K}_{1} h_{H}(P).
\]
Thus
\[
h_{H}(f^{nk}(P)) \leq \widetilde{K}_{1} \rho_{k}^{nk/k}h_{H}(P).
\]
By the same argument as in (Proof of Theorem \ref{FI} $\Longrightarrow $Theorem \ref{MI}),
we can prove the second statement.

\end{proof}

Now, we prove the convergence of canonical heights.

\begin{Prop}[Proposition \ref{main 5}]
Let $X$ and $f$ be as in Theorem \ref{picone}
Assume $f$ is algebraically stable and $\delta_{f}>1$.
Fix an ample height function $h_{X}$ on $X$.
Then
\[
\hat{h}_{X,f}(P)=\lim_{n\to \infty}\frac{h_{X}(f^{n}(P))}{\delta_{f}^{n}}
\]
exists for all $P \in X_{f}( \overline{\mathbb Q})$.
\end{Prop}

\begin{proof}
Since any ample heights are bounded below,
this follows from the following more general statement.
\end{proof}

\begin{Prop}
Let $X$ be a smooth projective variety over $ \overline{\mathbb Q}$.
Let $f \colon  X\dashrightarrow X$ be a dominant rational self-map defined over $ \overline{\mathbb Q}$.
Assume $\delta_{f}>1$ and there exists a nef $ {\mathbb{R}}$-divisor $H$ on $X$ such that $f^{*}H \equiv \delta_{f}H$.
Fix a height function $h_{H}$ associated with $H$.
Then for any $P \in X_{f}( \overline{\mathbb Q})$, 
the limit
\[
\lim_{n \to \infty} \frac{h_{H}(f^{n}(P))}{\delta_{f}^{n}}
\]
converges or diverges to $-\infty$.
\end{Prop}

\begin{proof}
We take a resolution of indeterminacy $p \colon Y \longrightarrow X$ of $f$ 
so that $p$ is an isomorphism outside the indeterminacy locus $I_{f}$ of $f$:
\[
\xymatrix{
&Y \ar[ld]_{p} \ar[rd]^{g}&\\
X \ar@{-->}[rr]_{f}&&X.
}
\]
Write $g=f\circ p$.
By negativity lemma, $p^{*}p_{*}g^{*}H-g^{*}H$ is a $p$-exceptional effective divisor on $Y$.
Then as in the proof of  \cite[Proposition 21]{ks}, we have
$h_{H}\circ f \leq h_{f^{*}H}+O(1)$ on $X\setminus I_{f}$
where $h_{H}$ and $h_{f^{*}H}$ are height functions associated with $H$ and $f^{*}H$.
Fix an ample height $h_{X}$ on $X$.
Since $f^{*}H\equiv \delta_{f}H$, we have $h_{f^{*}H}-\delta_{f}h_{H}=O\left( \sqrt[]{h_{X}^{+}}\right)$.
Thus, we have
\[
h_{H}\circ f\leq \delta_{f}h_{H}+O\left( \sqrt[]{h_{X}^{+}}\right) \quad \text{on $X\setminus I_{f}$}.
\]
Write $B=h_{H}\circ f-\delta_{f}h_{H}$.
Then, for any $P\in X_{f}$, 
\begin{align*}
h_{H}(f^{n}(x))&=\sum_{k=0}^{n-1}\delta_{f}^{n-1-k}\bigl(h_{H}(f^{k+1}(P))-\delta_{f}h_{H}(f^{k}(P))\bigr)+\delta_{f}^{n}h_{H}(P)\\
&=\sum_{k=0}^{n-1}\delta_{f}^{n-1-k}B(f^{k}(P))+\delta_{f}^{n}h_{H}(P).
\end{align*}
Take $\epsilon>0$ so that $ \sqrt[]{\delta_{f}+\epsilon}<\delta_{f}$.
By Theorem \ref{main}, there exists $C>0$ such that
$B(f^{k}(P))\leq C \sqrt[]{\delta_{f}+\epsilon}^{k}$ for all $k\geq0$.
Set 
\[
a_{k}=\frac{B(f^{k}(P))}{\sqrt[]{\delta_{f}+\epsilon}^{k}}.
\]
Note that $a_{k}$ is bounded above.
Then
\begin{align*}
\frac{h_{H}(f^{n}(P))}{\delta_{f}^{n}}&=h_{H}(P)+\sum_{k=0}^{n-1}\frac{B(f^{k}(P))}{\delta_{f}^{k+1}}\\
&=h_{H}(P)+\frac{1}{\delta_{f}}\sum_{k=0}^{n-1}a_{k}\left(\frac{ \sqrt[]{\delta_{f}+\epsilon}}{\delta_{f}}\right)^{k}\\
&=h_{H}(P)+\frac{1}{\delta_{f}}
\left\{
\sum_{\substack{0\leq k\leq n-1 \\ a_{k}\geq0}}a_{k}\left(\frac{ \sqrt[]{\delta_{f}+\epsilon}}{\delta_{f}}\right)^{k}
-\sum_{\substack{0\leq k\leq n-1 \\ a_{k}<0}}(-a_{k})\left(\frac{ \sqrt[]{\delta_{f}+\epsilon}}{\delta_{f}}\right)^{k}
\right\}.
\end{align*}
The first summation in the bracket
is convergent since $a_{k}$ is bounded above and the second summation is monotonically increasing.
Hence, the claim follows.

\end{proof}

\appendix
\section{lemmas}

\begin{Lem}\label{seq lem0}
Let $(a_{n})_{n\geq0}$ be a sequence of  positive real numbers with $a_{0}\geq1$
which satisfies
\[
a_{n}\leq a_{n-1}+C_{1}\left( \sqrt[]{a_{n-1}}+ \sqrt[]{a_{n-1}+C_{2}}\right)
\]
for all $n\geq1$.
Here $C_{1}, C_{2}$ are non-negative constants.
Then there exists a positive constant $\widetilde{C}$ depending only on $C_{1},C_{2}$ such that
\[
a_{n}\leq \widetilde{C}n^{2}a_{0}
\]
for all $n\geq1$.
\end{Lem}
\begin{proof}
Define a sequence $(b_{n})_{n\geq0}$ as follows.

\begin{align*}
&b_{0}=a_{0}\\
&b_{n}=b_{n-1}+C_{1}\left( \sqrt[]{b_{n-1}}+ \sqrt[]{b_{n-1}+C_{2}}\right)\ \ \ \ \text{for $n\geq 1$}.
\end{align*}
Then we have $a_{n}\leq b_{n}$.
By the definition, $(b_{n})_{n\geq0}$ is monotonically increasing.
In particular, $b_{n}\geq1$.
Thus
\[
b_{n}=b_{n-1}+C_{1} \sqrt[]{b_{n-1}}\left(1+ \sqrt[]{1+\frac{C_{2}}{b_{n-1}}}\right)
\leq b_{n-1}+C_{1}\left(1+ \sqrt[]{1+C_{2}}\right) \sqrt[]{b_{n-1}}.
\]
Let $C_{3}=C_{1}\left(1+ \sqrt[]{1+C_{2}}\right)$.

Define a sequence $(c_{n})_{n\geq0}$ as follows.

\begin{align*}
&c_{0}=b_{0}\\
&c_{n}=c_{n-1}+C_{3} \sqrt[]{c_{n-1}}\ \ \ \ \text{for $n\geq 1$}.
\end{align*}
Then we have $b_{n}\leq c_{n}$.
We take $\widetilde{C}$ so that 
\[
\widetilde{C}\geq \max\left\{\frac{{C_{3}}^{2}}{4}, 1+C_{3}\right\}.
\]
It is enough to show that $c_{n}\leq \widetilde{C}n^{2}c_{0}$ for $n\geq1$.
We prove this inequality by induction on $n$.
For $n=1$
\[
c_{1}=c_{0}+C_{3} \sqrt[]{c_{0}}\leq (1+C_{3})c_{0} \leq \widetilde{C}c_{0}.
\]
Assume $c_{n}\leq \widetilde{C}n^{2}c_{0}$.
Then
\begin{align*}
c_{n+1}&\leq \widetilde{C}n^{2}c_{0}+C_{3} \sqrt[]{\widetilde{C}n^{2}c_{0}}\\
&\leq \widetilde{C}n^{2}c_{0}+2 \sqrt[]{\widetilde{C}} \sqrt[]{\widetilde{C}n^{2}c_{0}}\\
&=\widetilde{C}\left(n^{2}+\frac{2n}{ \sqrt[]{c_{0}}}\right)c_{0}\\
&\leq \widetilde{C}(n^{2}+2n)c_{0}\\
&\leq \widetilde{C}(n+1)^{2}c_{0}.
\end{align*}

\end{proof}

\begin{Lem}\label{seq lem}
Let $(a_{n})_{n\geq 0}$ be a positive real sequence with $a_{0}\geq1$ which satisfies
\[
a_{n}\leq C(a_{0}+\sqrt[]{a_{0}}+\sqrt[]{a_{1}}+\cdots +\sqrt[]{a_{n-1}})\ \ \ \text{for all $n\geq1$}
\]
where $C$ is a positive constant.
For any $\widetilde{C}\geq1$ such that $\widetilde{C}\geq \max\{\frac{{C}^{2}}{4}, 1+C\}$,
we have 
\[
a_{n} \leq \widetilde{C}n^{2}a_{0}\  \ \ \ \text{for all $n\geq1$}.
\]
\end{Lem}
\begin{proof}
Let $(b_{n})_{n\geq 0}$ be a sequence such that 
\begin{align*}
&b_{0}=a_{0}\\
&b_{n}=C\left( b_{0}+ \sqrt[]{b_{0}}+\cdots + \sqrt[]{b_{n-1}}\right)\ \ \ \ \text{for all $n\geq1$}.
\end{align*}
Then clearly $a_{n}\leq b_{n}$ for all $n\geq0$.
By the definition of $b_{n}$, we have $b_{n+1}=b_{n}+C \sqrt[]{b_{n}}$.
Thus the statement follows from Lemma \ref{seq lem0} and its proof.
\end{proof}

\section{Positive characteristic}\label{pos chara}

In this section, we briefly remark how to modify the proof of Theorem \ref{FI} when the ground field has positive characteristic.
Let $K$ be an algebraically closed field with height function (e.g. $ \overline{ {\mathbb{F}}_{q}(t)}$ the algebraic closure of the function field over a finite field).

\begin{Prop}
Let $f \colon X \dashrightarrow Z$ be a dominant rational map of smooth projective varieties over $K$.
\begin{enumerate}
\item Let $Y$ be a projective variety with a birational morphism $p \colon Y \longrightarrow X$ and a 
morphism $g \colon Y \longrightarrow Z$ such that $f \circ p=g$.
For a Cartier divisor $D$ on $Z$, we define $f^{*}D=p_{*}[g^{*}D]$.
Here $[g^{*}D]$ is the codimension one cycle associated with the Cartier divisor $g^{*}D$. 
Then, the divisor $f^{*}D$ is independent of the choice of $Y$.
\item Let $\Gamma \subset X \times Z$ be the graph of $f$. For a Cartier divisor $D$ on $Z$, we have
$f^{*}D={{\rm pr}_{1}}_{*}({\rm pr}_{2}^{*}D \cdot \Gamma)$.
\item The map $f^{*}$ induces a homomorphism $f^{*} \colon N^{1}(Z) \longrightarrow N^{1}(X)$.
This definition of pull-back coincides the definition in \cite{dang, tru}.
\end{enumerate}
\end{Prop}

For a dominant rational self-map $f \colon X \dashrightarrow X$, let $p \colon Y \longrightarrow X$
be a blow-up of $X$ with a suitable ideal sheaf $ \mathcal{I}$ whose support is the indeterminacy locus $I_{f}$.
More precisely, take an embedding $i \colon X \longrightarrow {\mathbb{P}}^{N}$.
Then the linear system defining the morphism $i\circ f \colon X\setminus I_{f} \longrightarrow {\mathbb{P}}^{N}$
is uniquely extended to a linear system on $X$.
Then we can take $ \mathcal{I}$ to be the base ideal of this linear system.
Then there exists a surjective morphism $g \colon Y \longrightarrow X$ such that $g= f\circ p$.
Using this setting, we can argue as in the proof of Theorem \ref{FI}.

The only non-trivial point is the following.
In the proof, we need to bound height functions associated with numerically zero divisors.
Precisely, we need the inequality (\ref{rootboundE}).
On a smooth projective variety, this is well-known (see for example \cite{hs}).
Now we need this inequality on $Y$, which is possibly singular.
Actually, this inequality holds on any projective variety.

\begin{Lem}[see for example {\cite[Theorem 9.5.4]{kl}}]
Let $Y$  be a normal projective variety over an algebraically closed field.
Then there exists a morphism $\alpha \colon Y \longrightarrow A$ with $A$ is an Abelian variety with the following property.
For any line bundle $L$ on $Y$ which is algebraically equivalent to zero, there exists a line bundle $M$ on $A$
which is algebraically equivalent to zero such that $L \simeq \alpha^{*}M$. 
\end{Lem}

By this lemma and the argument in the proof of \cite[Theorem B.5.9]{hs}, we can easily prove the following.

\begin{Prop}\label{root bound on proj var}
Let $Y$ be a projective variety over $K$ and $E, H$ divisors on $Y$
with $E$ numerically equivalent to zero and $H$ ample.
Fix height functions $h_{E}, h_{H}$ associated with these divisors with $h_{H}\geq 1$.
Then there exists a positive constant $C>0$ such that
\[
|h_{E}| \leq C \sqrt{h_{H}}
\]
on $Y(K)$.
\end{Prop}

\begin{ack}
I would like to thank Tomohide Terasoma for giving me many valuable suggestions.
I thank Kaoru Sano for introducing me to the subject of the arithmetic degree
and for many helpful discussions.
I wish to thank Shu Kawaguchi for his comments and careful reading of a manuscript of this paper.
I thank Joseph H. Silverman for his comments. 
The author is supported by the Program for Leading Graduate Schools, MEXT, Japan.
\end{ack}

\end{document}